\documentclass[runningheads,envcountsame,a4paper]{llncs} 

\usepackage{amssymb,amsmath}
\usepackage{wasysym}
\usepackage{graphicx}
\usepackage{epsfig}
\usepackage{latexsym}
\usepackage[english]{babel}
\usepackage[utf8]{inputenc}
\usepackage{csquotes}
\usepackage{algorithm} 
\usepackage{algpseudocode}
\usepackage{tikz}
\usetikzlibrary{arrows,shapes.geometric,positioning,calc}
\usepackage[colorinlistoftodos]{todonotes}
\usepackage{changepage}
\usepackage{subfig}
\usepackage{soul}
\usepackage{hyperref}
\usepackage{todonotes}
\usepackage{enumerate}

\newtheorem{thm}{Theorem}
\newtheorem{prop}[thm]{Proposition}
\newtheorem{cor}[thm]{Corollary}
\newtheorem{lem}[thm]{Lemma}

\newtheorem{lemd}[thm]{Lemma and Notation}

\newsavebox{\qedB}

\sbox{\qedB}{\setlength{\unitlength}{1mm}
 \begin{picture}(4,4)(0,0)
  \thinlines
  {\put(0,0){\framebox(2.83,2.83){}}}%
  {\put(1.17,1.17){\framebox(2.83,2.83){}}}%
  {\put(0,0){\framebox(4,4){}}}%
  {\put(1.17,1.17){{\rule{1ex}{1ex} }}}%
 \end{picture}}
 
\newcommand{\QEDB}{\ifmmode\def\next{\tag"\usebox{\qedB}"}%
 \else\let\next=\relax
 {\unskip\nobreak\hfil\penalty50
 \hskip2em\hbox{}\nobreak\hfil\usebox{\qedB}
 \parfillskip=0pt \finalhyphendemerits=0\penalty-100\bigskip}\fi\next}

\newcommand{\Alphabet}{\hbox{\rm Alph}}

\newcommand{\fac}{\mathrm{Fac}}

\newcommand{\Class}{\mathrm{Class}}
\newcommand{\Index}{\mathrm{Index}}

\newcommand{\Powers}{\mathrm{Powers}}

\newcommand{\Res}{\mathrm{MaxPow}}
\renewcommand{\mp}{\mathrm{mp}}

\newcommand{\prim}{\mathrm{Prim}}
\newcommand{\bprop}{\begin{prop}}
\newcommand{\eprop}{\end{prop}}
\newcommand{\bcor}{\begin{cor}}
\newcommand{\ecor}{\end{cor}}
\newcommand{\blem}{\begin{lem}}
\newcommand{\elem}{\end{lem}}

\title{A note on the maximum number of $k$-powers in a finite word}
\author{Shuo Li\inst{1},
  Jakub Pachocki\inst{2}\fnmsep\thanks{Author participated in this research while a student at University of Warsaw, Poland.},
  Jakub Radoszewski\inst{3}\fnmsep\thanks{Supported by the Polish National Science Center, grant number 2018/31/D/ST6/03991.}}

\institute{Laboratoire de Combinatoire et d'Informatique Mathématique,\\
Université du  Québec \`a Montréal,\\
CP 8888 Succ. Centre-ville, Montréal (QC) Canada H3C 3P8\\
\email{shuo.li.ismin@gmail.com}
\and
Open AI, San Francisco, CA, USA\\
\email{jakub@openai.com}
\and
Institute of Informatics, University of Warsaw, Poland\\
\email{jrad@mimuw.edu.pl}
}

\begin{document}

\maketitle

%%%%%%%%%%%%%%%%%%%%%%%%%%%%%%%%%%%%%%%%%%%%%%%%%%%%
% Sections
%%%%%%%%%%%%%%%%%%%%%%%%%%%%%%%%%%%%%%%%%%%%%%%%%%%%

\begin{abstract}
A \emph{power} is a word of the form $\underbrace{uu...u}_{k \; \text{times}}$, where $u$ is a word and $k$ is a positive integer; the power is also called a {\em $k$-power} and $k$ is its {\em exponent}.
We prove that for any $k \ge 2$, the maximum number of different non-empty $k$-power factors in a word of length $n$ is between $\frac{n}{k-1}-\Theta(\sqrt{n})$ and $\frac{n-1}{k-1}$.
We also show that the maximum number of different non-empty power factors of exponent at least 2 in a length-$n$ word is at most $n-1$.
Both upper bounds generalize the recent upper bound of $n-1$ on the maximum number of different square factors in a length-$n$ word by Brlek and Li (2022).
\end{abstract}

\section{Introduction}
Let $k$ be an integer greater than $1$, the {\em $k$-power} (or simply the {\em power}) of a word $u$ is a word of the form $\underbrace{uu...u}_{k \; \text{times}}$. Here $k$ is called the {\em exponent} of the power. We consider only powers of non-empty words. A factor (subword) of a word is its fragment consisting of a number of consecutive letters. In this paper, we investigate the bounds for the maximum number of different $k$-power factors in a word of length $n$. This subject is one of the
fundamental topics in combinatorics on words~\cite{lothaire3}. For any pair of positive integers $(n,k)$ with $k >1$, let $N(n,k)$ denote the maximum number of different non-empty $k$-powers that can appear as factors of a word of length $n$. For 2-powers (squares), the bounds for $N(n,2)$ were studied by many authors; see~\cite{FraenkelS98,Ilie,lam,DezaFT15,thie,Brlekli}. The best known lower bound from \cite{FraenkelS98} and a very recent upper bound from \cite{Brlekli} match up to sublinear terms:
$$n-o(n) \leq N(n,2) \leq n-1.$$
Actually, one can check that the lower bound from \cite{FraenkelS98} is of the form $n-\Theta(\sqrt{n})$. For $k=3$, it was proved in~\cite{KUBICA} that
$$\frac12 n-2\sqrt{n} \leq N(n,3) \leq \frac{4}{5}n.$$
More generally, for $k\geq3$, it was studied in~\cite{li2022} and proved that
$$N(n,k) \leq \frac{n-1}{k-2},$$
with the same notation as above.
Further in \cite{KUBICA} it was shown that the maximum number of different factors of a word of length $n$ being powers of exponent {\em at least} 3 is $n-2$.
 
In this article, we generalize the methods provided in~\cite{Brlekli} and~\cite{KUBICA} to give an upper and a lower bound for the number of different $k$-powers in a finite word. The main result is announced as follows:

\begin{thm}
\label{th:main}
Let $k$ be an integer greater than 1. For any integer $n>1$, let $N(n,k)$ denote the maximum number of different $k$-powers being factors of a word of length $n$. Then we have $$\frac{n}{k-1}-\Theta(\sqrt{n})\leq N(n,k) \leq \frac{n-1}{k-1}.$$
\end{thm}

We also show the following result. It implies, in particular, that a word that contains powers of exponent greater than 2 has fewer squares than $n-1$.

\begin{thm}
\label{th:main2}
The maximum number of different factors in a word of length $n$ being powers of exponent at least 2 is $n-1$.
\end{thm}

\section{Preliminaries}

Let us first recall the basic terminology related to words. By a {\em word} we mean a finite concatenation of symbols $w = w_1 w_2 \cdots w_{n}$, with $n$ being a non-negative integer. The {\em length} of $w$, denoted $|w|$, is $n$ and we say that the symbol $w_i$ is at the {\em position} $i$. The set $\Alphabet(w)=\left\{w_i| 1\leq i \leq n\right\}$ is called the {\em alphabet} of $w$ and its elements are called {\em letters}. Let $|\Alphabet(w)|$ denote the cardinality of $\Alphabet(w)$. A word of length $0$ is called the {\em empty word} and it is denoted by $\varepsilon$. \emph{Concatenation} of two words $u$, $v$ is denoted as $uv$. %For any word $u$, we have $u=\varepsilon u=u\varepsilon$.

A word $u$ is called a {\em factor} of a word $w$ if $w = pus$ for some words $p$, $s$; $u$ is called a {\em prefix} ({\em suffix}) of $w$ if $p=\varepsilon$ ($s=\varepsilon$, respectively). The set of all factors of a word $w$ is denoted by $\fac(w)$.

Two words $u$ and $v$ {\em conjugate} when there exist words $x,y$ such that $u=xy$ and $v=yx$.
The conjugacy class of a word $v$ is denoted by $[v]$. If $v=v_1v_2\cdots v_m$ is word, then for any $i \in \{1,\ldots,m\}$, we define $v_p(i)=v_1v_2\cdots v_i$ and $v_{s}(i)=v_{i+1}v_{i+2}\cdots v_{m}$.
%For all $i$ such that $i > |v|$, let us define $v_s(i)=v_s(j)$ (resp. $v_p(i)=v_p(j)$) if $i \equiv j \; (\mod |v|)$. 
Thus, $[v]=\left\{v_s(i)v_p(i), i=1,2,\dots,m\right\}$.

For any positive integer $k$, we define the {\em $k$-power} (or simply a {\em power}) of a word $u$ to be the concatenation of $k$ copies of $u$, denoted by $u^k$. Here $k$ is the {\em exponent} of the power. A word $w$ is said to be {\em primitive} if it is not a power of another word, that is, if $w=u^k$ implies $k=1$. For any word $w$, there is exactly one primitive word $u$ such that $w=u^k$ for integer $k \geq 1$; the word $u$ is called the \emph{primitive root} of word $w$, see~\cite{lothaire1}. Furthermore, two words that conjugate are either both primitive or none of them is (\cite{lothaire1}). For a given word $w$, let $N_k(w)$ denote the number of different non-empty $k$-power factors of $w$ and $\prim(w)$ denote all primitive factors of $w$.

For any word $u$ and any rational number $\alpha$, the \emph{$\alpha$-power} of $u$ is defined to be $u^au_0$ where $u_0$ is a prefix of $u$, $a$ is the integer part of $\alpha$, and $|u^au_0|=\alpha |u|$. The $\alpha$-power of $u$ is denoted by $u^{\alpha}$. If $\alpha$ is a rational number greater than 1 and there exists a word $u$ such that $w=u^\alpha$, then the word $w$ is said to have a \emph{period} $|u|$.

\begin{lem}[Fine and Wilf~\cite{FineWilf}]
\label{perlemma}
Let $w$ be a word having $p$ and $q$ for periods. If $|w| \geq p + q - \gcd(p, q)$, then $\gcd(p, q)$ is also a period of $w$.
\end{lem}

\section{Rauzy graphs of a finite word}

In this section, we recall the notion of Rauzy graph and some results obtained in~\cite{Brlekli}. Let $w$ be a word of length $n$. For any integer $l \in \{1,\ldots,n\}$, let $L_l(w)$ be the set of all length-$l$ factors of $w$. For any integer $l \in \{1,\ldots,n\}$, let the Rauzy graph $\Gamma_l(w)$ be an oriented graph whose set of vertices is $L_l(w)$ and the set of edges is $L_{l+1}(w)$ (here $L_{n+1}(w)=\emptyset$);
an edge $e \in L_{l+1}(w)$ starts at the vertex $u$ and ends at the vertex $v$, if $u$ is a
prefix and $v$ is a suffix of $e$. %We remark that $\Gamma_l(w)$ is a weakly connected graph for any $l$.
Let us define $\Gamma(w)=\cup_{l=1}^{n}\Gamma_l(w)$.

Let $\Gamma_l(w)$ be a Rauzy graph of $w$. A sub-graph in $\Gamma_l(w)$ is called an {\em elementary circuit} if there are $j$ distinct vertices $v_1,v_2,\dots, v_j$ and $j$ distinct edges $e_1,e_2,\dots,e_j$ for some integer $j$, such that for each $t$ with $1 \leq t \leq j-1$, the edge $e_t$ starts at $v_t$ and ends at $v_{t+1}$, and for the edge $e_j$, it starts at $v_j$ and ends at $v_1$; further, $j$ is called the {\em size} of the circuit. The {\em small circuits} in the graph $\Gamma_l(w)$ are those elementary circuits whose sizes are no larger than $l$.

\begin{lemd}[Brlek and Li~\cite{Brlekli}]
\label{small-cycle}
Let $w$ be a word and let $\Gamma_l(w)$ be a Rauzy graph of $w$ for some $l \in \{1,\ldots,|w|\}$. Then for any small circuit $C$ on $\Gamma_l(w)$, there exists a unique primitive word $q$, up to conjugacy, such that $|q| \leq l$ and the vertex set of $C$ is $\left\{p^{\frac{l}{|p|}}| p \in [q]\right\}$ and its edge set is $\left\{p^{\frac{l+1}{|p|}}| p \in [q]\right\}$.

Further, each small circuit can be identified by an associated primitive word $q$ and an integer $l$ such that $\Gamma_l(w)$ is the Rauzy graph in which the circuit is located.
Let each small circuit be denoted by $C(q,r)$ with the parameters defined as above.
\end{lemd}

\begin{lem}[Brlek and Li~\cite{Brlekli}]
\label{bound}
Let $w$ be a word. Then there are at most $|w| -|\Alphabet(w)|$ small circuits in $\Gamma(w)$.
\end{lem}

\section{Upper bound for $N(n,k)$}
Let $w$ be a word and let $v \in \prim(w)$. A factor $u \in \fac(w)$ is said to be {\em in the class of factor $v$} if there is a (primitive) word $y \in [v]$ and an integer $p \geq 2$ such that $u=y^{p}$. Let $\Class_w(v)$ denote the set of all factors in the class of $v$. By $|\Class_w(v)|$ we denote the cardinality of $\Class_w(v)$. %Two classes $\Class_w(u)$ and $\Class_w(v)$ are equal if and only if $u$ and $v$ are conjugate.

For a factor $v$ of $w$, let us define $m_w(v)=\max\left\{n| v^{n} \in \fac(w), n \in \mathbb{N^+} \right\}$. Now given $\Class_w(v)$, let us define its {\em index} to be an integer $\Index_w(v)$ such that $\Index_w(v)=\max\left\{m_w(u)|u \in [v]\right\}$.
From the definition, the elements in $\Class_w(v)$ are all of the form $v_s(i)v^{j-1}v_p(i)$ with $1\leq i \leq |v|$ and $1\leq j \leq \Index_w(v)$.
By $\prim'(w)$ we denote the set of primitive words $v$ such that $v^{n}$ is in the class of $v$, where $n$ is the index of this class. In other words,
$$\prim'(w) = \{v \in \prim(w)| v^{\Index_w(v)} \in \Class_w(v)\}.$$
For $v \in \prim'(w)$, let $\Res_w(v)=\left\{u^{\Index_w(v)}| u \in [v] \right\} \cap \fac(w)$ and $\mp_w(v)$ denote the cardinality of $\Res_w(v)$.

\begin{example}\label{exclass}
Let $v=\mathtt{00001}$ and let us consider the following class of size 8 for some unspecified word $w$:
\begin{align*}
\Class_w(v)=\{&(\mathtt{00001})^2,(\mathtt{00010})^2,(\mathtt{00100})^2,(\mathtt{01000})^2,(\mathtt{10000})^2,\\
&(\mathtt{00001})^3,(\mathtt{00100})^3,(\mathtt{01000})^3\}.
\end{align*}
In this case, $\Index_w(v)=3$, $\Res_w(v)=\{(\mathtt{00001})^3,(\mathtt{00100})^3,(\mathtt{01000})^3\}$, and $\mp_w(v)=3$.
Moreover, the only words that conjugate with $v$ in $\prim'(w)$ are $\mathtt{00001},\mathtt{00100},\mathtt{01000}$.
\end{example}

\begin{lem}
\label{classes}
Let $u$ and $v$ be primitive words. If words $u$ and $v$ conjugate, then $\Class_w(u)=\Class_w(v)$. Otherwise, classes $\Class_w(u)$ and $\Class_w(v)$ are disjoint.
\end{lem}
\begin{proof}
The first part of the statement is obvious. Assume to the contrary that $y \in \Class_w(u) \cap \Class_w(v)$ for primitive words $u$ and $v$. This means that there exist words $u' \in [u]$ and $v' \in [v]$ and integers $k,t>1$ such that $y = (u')^k = (v')^t$. Word $y$ has periods $|u'|$ and $|v'|$, so by Lemma~\ref{perlemma} it has period $p=\gcd(|u'|,|v'|)$. If $p < |u'|$ ($p<|v'|$, respectively), then $p$ would divide $|u'|$ ($|v'|$, respectively); consequently, $u'$ (respectively $v'$) would not be primitive. Hence, $p=|u'|=|v'|$, so $u'=v'$ and $u$ and $v$ conjugate.
\qed\end{proof}

In this section we give an upper bound for $N_k(w)$. The strategy is as follows: first, we compute the exact number of powers in each class $\Class_w(v)$ of $w$; second, we prove that there exists an injection from $\cup_{v \in \prim'(w)} \Class_w(v)$ to the set of small circuits in $\Gamma(w)$; third, we conclude by using the proprieties of Rauzy graphs introduced in the previous section.

\begin{lem}
\label{number}
Let $w$ be a word and $v \in \prim'(w)$. If $\Index_w(v) \geq 2$, then we have $|\Class_w(v)|=|v|(\Index_w(v)-2)+\mp_w(v)$. Further, we have $\Class_w(v)=\left\{u^k| u \in [v], 2\leq k \leq \Index_w(v)-1\right\} \cup \Res_w(v)$.
\end{lem}

\begin{proof}
We only need to prove that for any $u \in [v]$ and any integer $k$ satisfying $2\leq k \leq \Index_w(v)-1$, $u^k \in \fac(w)$. We can easily check that $u^k \in \fac(v^{k+1})$. However, from the hypothesis, $k+1 \leq \Index_w(v)$ and $v^{\Index_w(v)} \in \fac(w)$, thus, $v^{k+1} \in \fac(w)$. Consequently, $u^k \in \fac(w)$. \qed
\end{proof}

Example~\ref{excycle} below shows the main ideas from the proof of the following lemma for a concrete class.

\begin{lem}
\label{cycle}
Let $w$ be a word and $v \in \prim'(w)$ such that $|v|=l$ and $\Index_w(v)\geq 2$. For any integer $t$ satisfying $1 \leq t \leq |\Class_w(v)|$, there exists a small circuit $C(v,t+l-1)$ in the Rauzy graph $\Gamma_{t+l-1}(w)$. Hence, there exists a bijective function $f_v$ which associates each word in $\Class_w(v)$ to a small circuit in the set $\left\{C(v, t+l-1)|1 \leq t \leq |\Class_w(v)|\right\}$. 
\end{lem}

\begin{proof}

To prove the existence of the circuit $C(v, t+l-1)$ for any integer $t \in \{1,\ldots,|\Class_w(v)|\}$, it is enough to prove that $$S_t:=\left\{u^{\frac{t+l}{l}}|u \in [v]\right\} \subset \fac(w).$$ If this is the case, then there exists a circuit in $\Gamma_{t+l-1}(w)$ such that its edge set is $S_t$; further, it can be identified by $C(v, t+l-1)$.

If integers $t',t$ satisfy $1 \leq t' < t \leq |\Class_w(v)|$, then each word in $S_{t'}$ is a prefix of a word in $S_t$. Hence, it is enough to prove that $S_t \subset \fac(w)$ for $t=|\Class_w(v)|$.
From Lemma~\ref{number}, we have $t=l(\Index_w(v)-2)+\mp_w(v)$, so $\frac{t+l}{l}=\Index_w(v)-1+\frac{\mp_w(v)}{l}$.

Let $i=\Index_w(v)$ and $j=\mp_w(v)$.
For any $u^{i-1+\frac{j}{l}} \in S_t$, the word $u^{i-1+\frac{j}{l}}=u^{i-1}u_p(j)$ is a factor of the word $u_s(m) u^{i-1} u_p(m)=(u_s(m)u_p(m))^i$ for all $m \in \{j,\ldots,l\}$.
Hence, there are at most $j-1$ distinct words $y$ that conjugate with $u$ such that $u^{i-1+\frac{j}{l}} \not\in \fac(y^{i})$.
In particular, there are at most $j-1$ distinct words $y^i\in \Res_w(v)$ which do not contain $u^{i-1+\frac{j}{l}}$ as a factor.
However, there are exactly $j$ elements in $\Res_w(v)$, so there exists at least one word in $\Res_w(v)$ containing $u^{i-1+\frac{j}{l}}$.
Thus, $S_t \subset \fac(w)$.

The existence of the bijective function $f_v$ is from the fact that the cardinalities of $\Class_w(v)$ and $\left\{C(v, t+|v|-1)|1 \leq t \leq |\Class_w(v)|\right\}$ are the same.\qed
\end{proof}

\begin{example}\label{excycle}
Let us consider the class from Example~\ref{exclass}. We need to show that all words in $S_8=\left\{u^{\frac{13}{5}}|u \in [v]\right\}$, for $v=\mathtt{00001}$, are factors of $w$. Let us consider $u^{\frac{13}{5}}=(\mathtt{00010})^{\frac{13}{5}}=\mathtt{0001000010000} \in S_8$. It is a factor of three words $y^3$, where $u$ and $y$ conjugate:
\begin{align*}
(\mathtt{10000})^3&=\mathtt{10} \underline{\mathtt{0001000010000}}\\
(\mathtt{00001})^3&=\mathtt{0} \underline{\mathtt{0001000010000}} \mathtt{1}\\
(\mathtt{00010})^3&=\underline{\mathtt{0001000010000}} \mathtt{10}
\end{align*}
and is not a factor of the two remaining such words. The set $\Res_w(v)$ contains three words and indeed $u^{\frac{13}{5}}$ is a factor of $y^3$ for $y$ being one of them, $y=\mathtt{00001}$.
\end{example}

For a word $w$, let $\Powers(w)$ denote the set of all powers of exponent at least 2 that are factors of $w$, i.e.\ $\Powers(w)=\{u^t|t \ge 2, u^t \in \fac(w)\}$.

\begin{lem}
\label{inj}
There exists an injective function $f$ from the set $\Powers(w)$ to the set of small circuits in $\Gamma(w)$.
\end{lem}

\begin{proof}
Each power factor of $w$ of exponent at least 2 belongs to some class. Hence, $\Powers(w) = \cup_{v \in \prim'(w)} \Class_w(v)$. For any $v \in \prim'(w)$, from Lemma~\ref{cycle}, there is a bijection $f_v$ from $\Class_w(v)$ to $\left\{C(v,t+|v|-1)|1 \leq t \leq |\Class_w(v)| \right\}$. Let us define the function $f$ as follows: for any $\Class_w(v)$, we set $f|_{\Class_w(v)}=f_v$ with $f_v$ defined as above. This function is well defined by Lemma~\ref{classes}. Now we prove that $f$ is injective. Let $y,z$ be two powers such that $f(y)=f(z)=C(v,t+|v|-1)$ for some $v$ and $t$. In this case, $y,z$ are both in $\Class_w(v)$. However, for a given class $\Class_w(v)$, $f_v$ is bijective, thus $y=z$. \qed
\end{proof}

\begin{proof}[of Theorem~\ref{th:main2}]
  The theorem can be stated as: $|\Powers(w)| \leq |w|-1$ for any word $w$.
  It is a direct consequence of Lemmas~\ref{inj} and~\ref{bound}.\ \ \qed
\end{proof}

\begin{thm}[Upper Bound]~\label{upper}\\
Let $k$ be an integer greater than 1. For any word $w$, we have $$N_k(w) \leq \frac{|w|-|\Alphabet(w)|}{k-1}.$$
Consequently, for any integer $n \geq 1$, we have $$N(n,k) \leq \frac{n-1}{k-1}.$$
\end{thm}

\begin{proof}
To each $k$-power factor in $\Powers(w)$ we can assign at least $k-2$ powers in the set $\Powers(w)$ that are not $k$-powers. More precisely, if the $k$-power factor is $v^{kp}$, for positive integer $p$ and primitive word $v$, then the words $v^{kp-1},\dots,v^{kp-k+2}$ are elements of $\Powers(w)$ and are not $k$-powers by uniqueness of primitive roots. (If $p>1$, we could also assign $v^{kp-k+1}$ to $v^{kp}$; however, for $p=1$ this would be a 1-power.) Moreover, this way the sets of powers assigned to different $k$-powers are disjoint. By Theorem~\ref{th:main2},
$$N_k(w)\leq \frac{\Powers(w)}{k-1} \leq \frac{|w|-|\Alphabet(w)|}{k-1}.\quad \qed$$ 
%Recall that for a primitive word $v$, $m_w(v)=\max\left\{n| v^{n} \in \fac(w), n \in \mathbb{N^+} \right\}$. For a given integer $k \geq 2$, by the fact that every word has a unique primitive root, we have
%$$N_k(w)=\sum_{\substack{v \in \prim(w)\\m_w(v) \geq k}}\left\lfloor \frac{m_w(v)}{k}\right\rfloor,$$ where $\left\lfloor x \right\rfloor$ represents the largest integer smaller than or equal to $x$. If $m_w(v) \geq k$, we have $\frac{m_w(v)}{k}\leq \frac{m_w(v)-1}{k-1}$. Hence,
%
%$$N_k(w) \le \sum_{\substack{v \in \prim(w)\\m_w(v) \geq k}} \frac{m_w(v)}{k} \leq \sum_{\substack{v \in \prim(w)\\m_w(v) \geq k}} \frac{m_w(v)-1}{k-1}.$$
%
%We have
%$$\sum_{v \in \prim(w)}m_w(v)-1 = |\Powers(w)| \leq |w|-|\Alphabet(w)|,$$
%where the inequality follows by Lemma~\ref{inj} and Lemma~\ref{bound}.
%We obtain the conclusion as follows:
%$$N_k(w)\leq \sum_{\substack{v \in \prim(w)\\m_w(v) \geq k}} \frac{m_w(v)-1}{k-1} \leq \sum_{v \in \prim(w)} \frac{m_w(v)-1}{k-1} \leq \frac{|w|-|\Alphabet(w)|}{k-1}.\quad \qed$$ 
\end{proof}

\section{Lower bound for $N(n,k)$}
  We show a family of binary words which yields a lower bound of $\frac{n}{k-1}-\Theta(\sqrt{n})$
  for the number of different factors which are $k$-powers, for an integer $k \geq 2$.

  For integers $i \geq 1$ and $k \geq 2$ we denote
  $$q^{(k)}_i = (\mathtt{1}\mathtt{0}^i)^{k-1}.$$
  Let $r^{(k)}_m$ be the concatenation
  $$r^{(k)}_m = q^{(k)}_1 q^{(k)}_2 \cdots q^{(k)}_m \mathtt{10}^m.$$
  E.g., for $k=2$, we obtain the family of words:
  $$\mathtt{1010},\ \mathtt{10100100},\ \mathtt{1010010001000},\ \mathtt{1010010001000010000},\ldots$$
  and for $k=3$, the family:
  $$\mathtt{101010},\ \mathtt{1010100100100},\ \mathtt{1010100100100010001000},\ldots$$

  \begin{lem}\label{l:lenqn}
    The length of $r^{(k)}_m$ is $(k-1)\left(\frac{m^2}{2} + \frac{3m}{2}\right) + m+1$.
  \end{lem}
  \begin{proof}%\mbox{ \ }\\
    The length of $q^{(k)}_i$ is $(k-1)(i+1)$, so
    $$|r^{(k)}_m| = \left(\sum_{i=1}^m (k-1)(i+1)\right)+m+1 = (k-1)\left(\frac{m^2}{2} + \frac{3m}{2}\right) + m+1.\ \ \qed$$
  \end{proof}

  \begin{lem}\label{l:cubesqn}
    $N_k(r^{(k)}_m) \geq \frac{m^2}2 + \frac{m}2 + \left\lfloor\frac{m}k\right\rfloor$.
  \end{lem}
  \begin{proof}%\mbox{ \ }\\
    Let us note that for a positive integer $i$, the concatenation $\mathtt{0}^{i-1}q^{(k)}_i\mathtt{1}\mathtt{0}^i = \mathtt{0}^{i-1} (\mathtt{1}\mathtt{0}^i)^k$
    contains as factors all the $k$-powers that conjugate with the $k$-power $(\mathtt{0}^i\mathtt{1})^k$
    that are different from this $k$-power.
    Let us note that this concatenation is a factor of $r^{(k)}_m$ for each $i \in \{1,\ldots,m\}$.
    Indeed, for $i \in \{1,\ldots,m\}$, the factor $q^{(k)}_i$ in $r^{(k)}_m$ is preceded by $\mathtt{0}^{i-1}$ (for $i=1$ this is the empty string, and otherwise it is a suffix of $q^{(k)}_{i-1}$)
    and followed by $\mathtt{1}\mathtt{0}^i$ (for $i<m$ it is a prefix of $q^{(k)}_{i+1}$, and for $i=m$ it is a suffix of $r^{(k)}_m$).

    Additionally, in $r^{(k)}_m$ there are $\left\lfloor\frac{m}k\right\rfloor$ unary $k$-powers $\mathtt{0}^k,\mathtt{0}^{2k},\ldots$
    In total we obtain
    $$\left(\sum_{i=1}^m i\right) + \left\lfloor\frac{m}k\right\rfloor =
      \frac{m^2}2 + \frac{m}2 + \left\lfloor\frac{m}k\right\rfloor$$
    $k$-powers, all pairwise different.
  \qed\end{proof}

  \begin{thm}[Lower Bound]~
\label{lower}
    Let $k\geq 2$ be an integer.
    For infinitely many positive integers $n$ there exists a word $w$ of length $n$
    for which $N_k(w) > \frac{n}{k-1}-2.2\sqrt{n}$.
  \end{thm}
  \begin{proof}%\mbox{ \ }\\
    Due to Lemmas~\ref{l:lenqn} and~\ref{l:cubesqn}, for any word $r^{(k)}_m$ we have:
    \begin{align*}
        \frac{|r^{(k)}_m|}{k-1} - N_k(r^{(k)}_m)  &\leq \frac{m^2}2 + \frac{3m}2 + \frac{m+1}{k-1} - \frac{m^2}2 - \frac{m}2 - \left\lfloor\frac{m}{k}\right\rfloor \\
        &= m + \frac{m+1}{k-1} - \left\lfloor\frac{m}{k}\right\rfloor \leq m + \frac{m+1}{k-1} - \frac{m-k+1}{k}\\
        &=m + \frac{m+1}{k(k-1)} + 1 \leq m+\frac{m+1}{2}+1 = \frac32m + \frac32.
    \end{align*}
    This value is smaller than $c \sqrt{|r^{(k)}_m|}$ for $c^2 \ge \frac92$; indeed, in this case, we have:
    $$\left(\frac32m + \frac32\right)^2 = \frac94m^2 + \frac92m + \frac94 < c^2\left(\frac12m^2 + \frac52m + 1\right) \leq c^2|r^{(k)}_m|.$$
    Hence, for $c \ge 2.2$ we conclude that:
    $$
      \frac{|r^{(k)}_m|}{k-1} - N_k(r^{(k)}_m) < c\sqrt{|r^{(k)}_m|} \quad\Rightarrow\quad
      N_k(r^{(k)}_m) > \frac{|r^{(k)}_m|}{k-1} - c\sqrt{|r^{(k)}_m|}.\ \ \qed
    $$
  \end{proof}

  \paragraph{\bf Note.}
  For $k=2$ we obtain a family of words containing $n-o(n)$ different squares
  that is simpler than the example by Fraenkel and Simpson~\cite{FraenkelS98}:
  we concatenate the words $q^{(2)}_i = \mathtt{10}^i$
  whereas they concatenate the words $q'_i = \mathtt{0}^{i+1}\mathtt{10}^i\mathtt{10}^{i+1}\mathtt{1}$.
  
  \begin{proof}[of Theorem~\ref{th:main}]
  It is a direct consequence of Theorems~\ref{upper} and~\ref{lower}.\ \ \qed
  \end{proof}

%%%%%%%%%%%%%%%%%%%%%%%%%%%%%%%%%%%%%%%%%%%%%%%%%%%%
% Biblio
%%%%%%%%%%%%%%%%%%%%%%%%%%%%%%%%%%%%%%%%%%%%%%%%%%%%

\bibliographystyle{splncs03}
\bibliography{biblio}

\end{document}